\documentclass[leqno,a4paper,11pt]{article}
\usepackage{amsmath,amsthm,amssymb}
\usepackage[utf8]{inputenc}
\usepackage[T1]{fontenc}
\usepackage[dvipsnames]{xcolor}

\usepackage{enumerate, xparse}
\usepackage{bbm}
\usepackage{todonotes}
\usepackage{mathrsfs}
\usepackage{mathabx}
\usepackage{hyperref}
\usepackage{nicefrac}
\usepackage{tikz}
\usetikzlibrary{matrix}

\usepackage{url}
\makeatletter
\g@addto@macro{\UrlBreaks}{\UrlOrds}
\makeatother

\usepackage[sort,numbers]{natbib}

\providecommand{\noopsort}[1]{} 

\usepackage{geometry}

\hypersetup
{
	colorlinks = true,
	linkcolor = {Green},
	anchorcolor = {black},
	citecolor = {Green},
	filecolor ={cyan},
	menucolor= {blue},
	runcolor ={cyan - same as file color},
	urlcolor= {magenta},
}
\usepackage[framemethod=TikZ]{mdframed}
\usepackage[most]{tcolorbox}

\usepackage{cleveref}
\newtcbtheorem{conj}{Conjecture}{label type=conj}{conj}

\newtcbtheorem[auto counter, number within = section,
crefname={theorem}{theorems},
Crefname={Theorem}{Theorems} ]
{theo}{Theorem}{%
	breakable =false,
	fonttitle = \bfseries,
	colframe = Green,
	colback = White,
	colbacktitle=White,
	description color=Black,
	coltitle=Green,
	colupper=Black
}{theo}

\newtcbtheorem[use counter from=theo, number within = section,
crefname={definition}{definitions},
Crefname={Definition}{Definitions} ]
{defi}{Definition}{%
	breakable =false,
	fonttitle = \bfseries,
	colframe = Green,
	colback = White,
	colbacktitle=White,
	description color=Black,
	coltitle=Green,
	colupper=Black
}{defi}
\newtcbtheorem[use counter from=theo, number within = section,
crefname={remark}{remarks},
Crefname={Remark}{Remarks} ]
{rema}{Remark}{%
	breakable =false,
	fonttitle = \bfseries,
	colframe = Green,
	colback = White,
	colbacktitle=White,
	description color=Black,
	coltitle=Green,
	colupper=Black
}{rema}

\newtcbtheorem[use counter from=theo, number within = section,
crefname={proposition}{propositions},
Crefname={Proposition}{Propositions} ]
{prop}{Proposition}{%
	breakable =false,
	fonttitle = \bfseries,
	colframe = Green,
	colback = White,
	colbacktitle=White,
	description color=Black,
	coltitle=Green,
	colupper=Black
}{prop}

\newtcbtheorem[use counter from=theo, number within = section,
crefname={corollary}{corollaries},
Crefname={Corollary}{Corollaries} ]
{coro}{Corollary}{%
	breakable =false,
	fonttitle = \bfseries,
	colframe = Green,
	colback = White,
	colbacktitle=White,
	description color=Black,
	coltitle=Green,
	colupper=Black
}{coro}

\newtcbtheorem[use counter from=theo, number within = section,
crefname={lemma}{lemmas},
Crefname={Lemma}{Lemmas} ]
{lemm}{Lemma}{%
	breakable =false,
	fonttitle = \bfseries,
	colframe = Green,
	colback = White,
	colbacktitle=White,
	description color=Black,
	coltitle=Green,
	colupper=Black
}{lemm}

\newtcbtheorem[use counter from=theo, number within = section,
crefname={example}{examples},
Crefname={Example}{Examples} ]
{exam}{Example}{%
	breakable =false,
	fonttitle = \bfseries,
	colframe = Green,
	colback = White,
	colbacktitle=White,
	description color=Black,
	coltitle=Green,
	colupper=Black
}{exam}
\newtheorem{Th}{Theorem}[section]

\theoremstyle{definition}

\newcommand{\beq}{\begin{equation}}
	\newcommand{\eeq}{\end{equation}}

\def\scalar(#1,#2){(#1\mid#2)}

\newcommand{\raz}{\mathbbm{1}}

\newcommand{\zdk}{(Z,{\cal D},\kappa)}
\newcommand{\ot}{\otimes}
\newcommand{\ov}{\overline}
\newcommand{\la}{\lambda}

\newcommand{\R}{{\mathbb{R}}}

\newcommand{\C}{{\mathbb{C}}}
\newcommand{\Z}{{\mathbb{Z}}}
\newcommand{\N}{{\mathbb{N}}}
\newcommand{\E}{{\mathbb{E}}}
\newcommand{\EE}{{\mathbb{E}}}

\newcommand{\vep}{\varepsilon}

\newcommand{\mob}{\boldsymbol{\mu}}
\newcommand{\lio}{\boldsymbol{\lambda}}

\newcommand{\bfu}{\boldsymbol{u}}

\newcommand{\bfv}{\boldsymbol{v}}

\NewDocumentCommand{\pof}{O{} O{} O{P}}{\mathbb{#3}_{#2}\left(#1\right)}
\NewDocumentCommand{\X}{}{\boldsymbol{X}}
\NewDocumentCommand{\x}{}{\boldsymbol{x}}
\NewDocumentCommand{\y}{}{\boldsymbol{y}}
\NewDocumentCommand{\Y}{}{\boldsymbol{Y}}
\NewDocumentCommand{\seq}{m m O{\in} m}{\left(#1_{#2}\right)_{#2  #3 #4}}
\newcommand{\ds}[1][X]{\boldsymbol{\mathcal{#1}}}

\newcommand{\indep}{\perp \!\!\! \perp}
\renewcommand{\amalg}{\indep}

\NewDocumentCommand{\di}{O{p} O{q}}{\mathbb{D}\left(#1 \;\middle|\middle|\; #2\right)}

\NewDocumentCommand{\h}{O{X} O{}}{\mathbb{H}_{#2}\left(#1\right)}
\NewDocumentCommand{\ph}{O{X} O{}}{\mathbb{H}_{#2}\left(\mb{#1}\right)}
\NewDocumentCommand{\pch}{O{X} O{Y}}{\mathbb{H}\left(\mathb{#1} | \mathbf{#2}\right)}
\NewDocumentCommand{\coh}{O{X} O{Y} O{}}{\mathbb{H}_{#3}\left(#1\;|\; #2\right)}
\NewDocumentCommand{\cph}{O{X} O{Y}}{\mathbb{H}\left(\mb{#1}\;|\; \mb{#2}\right)}
\NewDocumentCommand{\mc}{m}{\mathcal{#1}}
\NewDocumentCommand{\cmv}{O{X} O{Y} O{}}{\E_{#3}\left(#1 \middle| #2\right)}
\NewDocumentCommand{\cpof}{O{X} O{Y} O{}}{\mathbb{P}_{#3}\left(#1 \middle| #2\right)}

\newcommand{\iof}[1][]{\mathbf{1}_{#1}}
\NewDocumentCommand{\norm}{O{p} O{TV}}{\left\|#1\right\|_{#2}}

\NewDocumentCommand{\de}{m}{{\em \textbf{#1}}}

\begin{document}
\title{A note on Sarnak processes}
\author{Mariusz\ Lema\'nczyk, Michał D. Lema\'nczyk, Thierry de la Rue}

\maketitle


\begin{abstract}
   Basic properties of stationary processes called Sarnak processes are studied. As an application, a combinatorial reformulation of Sarnak's conjecture on M\"obius orthogonality is provided.
\end{abstract}
\section{Introduction}

    \subsection{Notation and definitions}
        Throughout this paper (unless stressed otherwise) random variables are defined on the probability space $(\Omega, \mathcal{F}, \mathbb{P})$. We consider double sided stochastic processes $\X = \seq{X}{i}{\Z} \in \ds$, where $X_i \in \mathcal{X}$ belong to the common \emph{state space} $\mathcal{X}$ and $\ds$ stands for the state space of $\X$. We assume that the common state space $|\mc{X}|$ is \textbf{finite} (and $\ds\subset \mc{X}^{\Z}$). For the convenience sake, for any sequence $\x = \seq{x}{i}{\Z}$ and $m, n \in \Z$ such that $m \le n$, we will write
        \begin{equation}
           x_{\le n} = \seq{x}{i}[\le]{n}, \qquad x_{< n} = \seq{x}{i}[<]{n} ,\qquad x_m^n = \left(x_i\right)_{m \le i \le n}
        \end{equation}
        Moreover, we extend each function or operation acting on $\mathcal{X}$ to the sequences by applying it coordinate-wise so that, e.g., for any sequences  $\x = \seq{x}{i}{\Z}$ and $\y = \seq{y}{i}{\Z}$, we have
        \begin{equation}
           (\x, \y) = \left(x_i, y_i\right)_{i \in \Z}, \qquad \x \y = \left(x_iy_i\right)_{i \in \Z} \qquad \text{}
        \end{equation}
the latter when $\x$ and $\y$ are complex-valued.

        We denote by $S: \mc{X}^\Z \to \mc{X}^\Z$ the left shift map acting on sequences $\x$ via $S\x = \left(x_{i + 1}\right)_{i \in \Z}$. We say that a process $\X$ is stationary if $\X$ and $S\X$ share the same distribution $\nu$ (for short $\X \sim S \X \sim \nu$). Recall that with every stationary processes we can associate the corresponding dynamical system $(\ds, \nu, S)$ where $\X \in \ds$, $\X \sim \nu$. Conversely, with every given dynamical system $(\ds, \nu, T)$ and a function $f:\ds \to \mathcal{Y}$, we can associate a stationary process (with $\Omega = \ds$), $\Y = \seq{Y}{i}{\Z}$, where $Y_i := f \circ T^i \in \mathcal{Y}$.

         Remembering that $\mc{X}$ is finite, we denote by $\h[X] =  -\sum_{x \in \mc{X}} \pof[X = x]\log \pof[X = x]$ the \de{Shannon entropy} of $X$ (note that in this paper we use base $2$ version of logarithm). If additionally an arbitrary random variable $Y$ is given such that the regular conditional probability $p_{X|Y}$ exists, $\coh[X][Y] = \E_Y \h[X^{(Y)}]$, stands for the \de{conditional Shannon entropy} ($\E_Y$ is the integration with respect to $Y$) and $X^{(y)} \sim p_{X|Y}(\cdot | y)$. For a given stationary process $\X$, $\h[\X] = \coh[X_0][X_{\le - 1}]$ stands for the \de{process entropy}. We say that process is \de{deterministic} (or of\de{ zero entropy}) if $\h[\X] = 0$. Similarly, given a dynamical system $(\ds[X], \nu, T)$ and a finite (measurable) partition $\ds[A] = (\boldsymbol{A}_i)_i$ of $\ds$ we denote by $\h[\nu, T, \ds[A]] = \h[\X_{\ds[A]}]$ the metric entropy corresponding to the partition $\ds[A]$ (here $\X_{\ds[A]}$ stands for the process determined by $f(\x) = i$ iff $\x \in \boldsymbol{A}_i$). For the (topological and metric) entropy theory of dynamical systems (and relations with the entropy of stationary processes), we refer the reader to \cite{Do}, \cite{Wa}.

        The \emph{past tail $\sigma$-algebra} (or \emph{remote past}) of $\X$ is defined by
        \begin{equation}
           \Pi(\X) := \bigcap_{i \ge 0} \sigma(X_{\le -i}).
        \end{equation}
        \begin{rema}{Pinsker $\sigma$-algebra}{}
           The past tail $\sigma$-algebra of $\X$ coincides with the the Pinsker $\sigma$-algebra of the corresponding dynamical system. Recall that the Pinsker $\sigma$-algebra is the largest invariant $\sigma$-algebra such that the corresponding factor of the dynamical system determined by $\nu \sim \X \in \ds$ has zero entropy. In fact, if for any measurable $\boldsymbol{B}$ we denote by $\ds[A]_{\boldsymbol{B}}$ the binary partition generated by $\boldsymbol{B} \subset \ds$ then
           \beq\label{defPi}
           \Pi(\X) = \sigma\left\{\boldsymbol{B} \colon  \h[\nu, S, \ds[A]_{\boldsymbol{B}}]=0\right\},\eeq
        \end{rema}

        In connection with the celebrated Sarnak's conjecture on M\"obius orthogonality (see below), in \cite{Ka-Ku-Le-Ru}, the following notion of Sarnak process has been introduced.

        \begin{defi}{Sarnak process}{d:SP}  A stationary process $\X$ such that $X_i \in \mc{X} \subset \C$, $|\mc{X}| < \infty$,  is called a {\em \textbf{Sarnak process}} if $\cmv[X_0][X_{\le -n}] \xrightarrow{n\to\infty} 0$ in $L^2$. Equivalently, $\cmv[X_0][\Pi(\X)] = 0$.
        \end{defi}
        \begin{rema}{Basic properties of a Sarnak process}{}
           Sarnak processes are centered, that is $\E X_0 = 0$. Moreover, every non-zero Sarnak process is, by the very definition, of positive entropy but it need not be even ergodic\footnote{In fact, the dynamical system given by $\X$ is ergodic if and only if so is its Pinsker factor.} (see Section~\ref{s:BasicProp}). On the other hand, we will show that the class of Sarnak processes is stable under multiplication by deterministic processes.
        \end{rema}

\section{Results}
    In this section we assume that every process $\X$ is \textbf{stationary} and $X_i \in \mc{X} \subset \C$, where $|\mc{X}| < \infty$. Sometimes additionally we assume that $
       \mc{X} = \{-1, 1\}$. Note that in such a case, if $\E X_0 = 0$ (in particular this holds if $\X$ is Sarnak), then $\pof[X_0 = \pm 1] = 1/2$ and hence $X_i$'s are random signs. In this section we present a bunch of results. The most important ones are formulated in~\Cref{theo:p:alaweakB} and~\Cref{theo:p:thierry}. The first one can be viewed as a number theoretic reformulation of Sarnak's conjecture. The second one is of slightly different flavor, namely, it shows that every deterministic process can be embedded as the Pinsker factor in some Sarnak process. Before we present the theorems, let us recall the formulation of Sarnak's conjecture.

    Given an arithmetic function $\bfu:\N\to \mathcal{X}$, where $\mathcal{X}\subset\C$ is finite, and  whose mean
    $$M(\bfu):=\lim_{n\to\infty}\frac1n\sum_{i\leq n}u_i$$ exists and equals $0$, we ask whether
    \beq\label{Sarnakg}
    \lim_{n\to\infty}\frac1n\sum_{i\leq n}f(T^ix)u_i = 0\eeq
    for each zero entropy topological dynamical system $(X,T)$, all $f\in C(X)$ and all $x\in X$. The most known instance of~\eqref{Sarnakg} is  Sarnak's conjecture \cite{Sa} which predicts that \eqref{Sarnakg} holds for $\bfu=\lio$, where $\lio=(\lambda_n)$ stands for the Liouville function ($\lambda_n=1$ if $n$ has an even number of primes divisors counted with multiplicity, and -1 otherwise).
    \begin{rema}{Logarithmic Sarnak's conjecture}{}
       The logarithmic version, originated by Tao \cite{Ta1}, in which instead of Ces\`aro averages $1/n\sum_{i\leq n}a_i$ we consider $1/ \log n\sum_{i\leq n}a_i/i$ is equally interesting. In fact, Tao proved that the logarithmic Sarnak's conjecture is equivalent to the logarithmic Chowla conjecture.
    \end{rema}
    \begin{rema}{}{}
       Originally, Sarnak's conjecture has been formulated for the M\"obius function, $\bfu=\mob$, but it is known that Sarnak's conjecture for $\lio$ and $\mob$ are equivalent, \cite{Fe-Ku-Le}, \cite{Ta1}.
    \end{rema}
    \subsection{Main results}
        Our first result links the notion of Sarnak process with the classical concept of weakly Bernoulli processes (see e.g.\ \cite{Sh}). Recall that  a stationary process $\X \sim \nu$ is \de{weakly Bernoulli}  if, given $\vep>0$, there is a gap $g\geq0$ such that for each $n\geq1$ and $m\geq0$,
        $$
           \E_{X_{-g-m}^{-g}}\norm[\cpof[X_{0}^{n - 1} \in \cdot][X_{-g-m}^{-g}] - \pof[X_{0}^{n - 1} \in \cdot]]<\vep,$$
        where $\E_{X_{-g-m}^{-g}}$ denotes the integration with respect to ${X_{-g-m}^{-g}}$ and, for any complex measure $\mu$ defined on a $\sigma$-algebra $\mc{A}$,
        \begin{equation}
           \norm[\mu] := \sup_{A \in \mc{A}} |\mu(A)| + |\mu(A^c)|
        \end{equation}
        stands for the total variation of measure $\mu$.
        It turns out that in the special case of $X_i \in \{-1, 1\}$, if we do not require in the above condition uniformity in $n$ then we recover the class of Sarnak processes.

        \begin{theo}{Total variation formulation of the Sarnak property}{p:alaweakB}
           Let $\X$ be a stationary process taking values $\pm1$. Then $\X$ is Sarnak if and only if, given $\vep>0$, there is a gap $g\geq0$ such that for  all $m\geq0$,
           \beq\label{WB1}
           \E_{X_{-g-m}^{-g}}\norm[\cpof[X_{0} \in \cdot][X_{-g-m}^{-g}] - \pof[X_{0} \in \cdot]]<\vep.\eeq
        \end{theo}

        Let us give some intuition  behind \eqref{WB1}. Think about large gap $g$. Then \eqref{WB1} can be roughly be interpreted as follows. Almost none of typical realizations of block in the far past $X_{-g-m}^{-g} = x_{-g-m}^{-g}$ affects the probability of occurrence of given sign at a given non-negative coordinate.  Note that  Sarnak's conjecture is of similar flavour, namely, it roughly states that the Liouville function (counterpart of our $\X$) is uncorrelated with deterministic sequences (thus, in some sense, Liouville function behaves like a random sign). Thus, there is no surprise that as a corollary of \Cref{theo:p:alaweakB}, we obtain the following description of Sarnak's conjecture.
        %
        %
        %
        %

        \begin{coro}{Number-theoretic reformulation of Sarnak's conjecture}{t:main}
           Let $\bfu:\N\to\pm1$. Then $\bfu$ satisfies Sarnak's conjecture (i.e.~it satisfies~\eqref{Sarnakg}) if and only if for every increasing sequence $\seq{n}{k}{\N} \in \N^\N$ such that  $1/ n_k \sum_{0 \le n\leq n_k - 1}\delta_{S^n\bfu}\overset{k \to \infty}{\Rightarrow}\kappa$ and for each $\vep>0$ there exists a gap $g\geq1$ such that for all $m \ge 0$,
           $$
              \sum_{q\in\{-1,1\}^m}\vep_{g}(q)<\vep,$$
           where for any block $q \in \{-1, 1\}^m$, $$
              \vep_{g}(q):=\lim_{k\to\infty}\frac1{n_k}\Big||\{n + g\in A(q,n_k) \colon u_n =1\}|-\frac12|A(q,n_k)|\Big|$$
           and $
              A(q,n) :=\{i\leq n\colon u_i^{i + m -1}=q\}$
           is the set of appearances of the block $q$ in $u_n^{n + m -1}$.

        \end{coro}
        \begin{rema}{Magical constant $1/2$ in the definition of $\vep_{g}$}{}
           Note that the constant $1/2$ appearing in the definition of $\vep_{g}$ corresponds to the fact that we randomly choose a sign ($\vep$-independently of the given remote past).
        \end{rema}
        \begin{rema}{Logarithmic version}{}
           Logarithmic version of \Cref{coro:t:main} holds as well. It suffices to replace all arithmetic weights  by their logarithmic counterparts.
        \end{rema}

        One may wonder under what assumptions put on $\Y$ the following statement holds: if $\X$ is Sarnak then so is $\X \Y$. The following theorem states that this statement is valid if we consider the class of zero entropy processes.
        \begin{theo}{Closure under multiplication by a deterministic process}{t:mnozenie}
           Let $(\X, \Y) = \left(X_i, Y_i\right)_{i \in \Z}$ be a stationary process. If $\X$ is Sarnak and $\Y$ is deterministic, then $\X\Y$ is  Sarnak.
        \end{theo}


        The next fact states that each deterministic system can be realized as the Pinsker factor of some Sarnak process. It shows as well that Sarnak processes $\X$ (even restricted to the case of $\mathcal{X}=\{-1,1\}$) are ubiquitous.
        \begin{theo}{Embedding of a deterministic process into a Sarnak one}{p:thierry}
           For each aperiodic zero entropy dynamical system $R$ there is a $\pm1$-valued Sarnak process $\X$ such that the Pinsker factor of the dynamical system given by $\X$ is (measure-theoretically) isomorphic to $R$.
        \end{theo}

While Sarnak's processes need not be ergodic, see e.g.\ \Cref{exam:ex:non_erg}, in fact, we have the following result.
\begin{theo}{}{t:EDM} Let $\X$ be a stationary process. Then $\X$ is Sarnak if and only if the stationary processes in its ergodic decomposition are a.a.\ Sarnak.\end{theo}

\begin{coro}{}{c:EDT} Let $\X$ be a Sarnak process. Then for a.a.\ realizations $(X_n(\omega))$ of the process, the sequence $(X_n(\omega))$ is orthogonal to all topological dynamical systems of zero entropy.\end{coro}

\section{Motivation}

    \subsection{Chowla conjecture}
        Sarnak's conjecture is motivated by the Chowla conjecture \cite{Ch} from 1965 about vanishing of all auto-correlations of $\lio$:
        \beq\label{Chowla}
        \lim_{n\to\infty} \frac1n\sum_{i\leq n}\lambda_{n}\lambda_{n+r_1}\cdots\lambda_{n+r_k}=0\eeq
        for each choice of $1\leq r_1<\ldots<r_k$.
        \begin{rema}{Chowla conjecture implies Sarnak's conjecture}{}
           The implication \eqref{Chowla}~$\Rightarrow$~\eqref{Sarnakg} (for $\bfu=\lio$) has been shown by Sarnak.
           See also the post \cite{Ta} for the original proof by Sarnak and \cite{Ab-Ku-Le-Ru} for a dynamical proof of this implication.
        \end{rema}
        It was noticed by Sarnak \cite{Sa} that Chowla conjecture has a dynamical reformulation, indeed,~\eqref{Chowla} is equivalent to the fact that $\lio$ is a generic point for the Bernoulli measure $\mc{B}\left(1/2, 1/2\right)$ on the full shift $\{-1,1\}^{\Z}$, see \cite{Ab-Ku-Le-Ru}, \cite{Fe-Ku-Le}.
    \subsection{Why Sarnak processes are interesting from the analytic number theory point of view?}
        As we have seen the Chowla conjecture can be seen from both: number-theoretic (combinatorial) and dynamical points of view, while Sarnak's conjecture, by its nature, is a purely dynamical statement. The natural question arises:
        \beq\label{pytanie}
        \mbox{\em Is there a number-theoretic reformulation of Sarnak's conjecture?}
        \eeq
        This question has been partially answered in \cite{Ka-Ku-Le-Ru}, where Veech's conjecture has been proved. Namely, it has been proved that Sarnak's conjecture is equivalent to the fact that for each Furstenberg system $(X_{\lio},\kappa,S)$ of $\lio$, the corresponding stationary process $\boldsymbol{\pi}=(\pi_n)$ is Sarnak (in the sense of Definition~\ref{defi:d:SP}), where $\pi_n(\x) = x_n$.
        \begin{rema}{Dynamical approach (Furstenberg systems), see e.g.\ \cite{Fr-Ho} and the surveys \cite{Fe-Ku-Le}, \cite{Ku-Le}}{}
           Let us only mention that in this approach through Furstenberg systems  $\bfu$ is treated as a two-sided sequence, for example via $u_{-n} := u_n$ and we consider the subshift $X_{\bfu}:=\ov{\{S^k\bfu\colon k\in\Z\}}\subset \mathcal{X}^{\Z}$. The set of shift invariant measures $\kappa$ obtained as weak$^\ast$-limits of $\frac1{n_k}\sum_{i\leq n_k}\delta_{S^i\bfu}$, $k\geq1$, is denoted by $V(\bfu)$ and each such measure makes the coordinate projection process $\boldsymbol{\pi}=(\pi_0\circ S^n)$ stationary.

        \end{rema}
        \begin{rema}{Mixing property approach, see \cite{Ka-Ku-Le-Ru}}{}
           In this approach the condition $\pi_0\perp L^2(\Pi(\boldsymbol{\pi}))$ is equivalent to a (``relative'') uniform mixing property  of the function $\pi_0$ which finally yields a certain combinatorial (hardly satisfactory) condition on $\bfu$. Let us add that the idea is borrowed from the classical ergodic theory, namely, we use the classical fact that so called Kolmogorov automorphisms are those satisfying the uniform mixing property, see e.g.\ \cite{Co-Fo-Si}.
        \end{rema}

        The main aim of this note is to consider the special case $\bfu \in \{-1, 1\}^\Z$. Sarnak processes resulting from this situation satisfy $\pi_0\amalg \Pi(\boldsymbol{\pi})$ which allows us to use some entropy techniques and to obtain
        a rather clear combinatorial condition given in \Cref{coro:t:main}. Note finally that if in \Cref{coro:t:main}, given a Furstenberg system $\kappa\in V(\bfu)$, for some $g\geq1$ we have $\vep_g(q)=0$ for all blocks  $q$ then for the stationary process $(\pi_n)$, we have $\pi_0
        $ is independent from the $\sigma$-algebra generated by $(X_g,X_{g+1},\ldots)$. In particular, the Chowla conjecture holds if and only if $\vep_1(q)=0$ for all blocks $q$.

\section{Examples of Sarnak processes}
    Firstly, let us make some remarks about properties of the class of Sarnak processes. It follows from the very definition that for each Sarnak process $\X=(X_n)$, we have $\EE X_0=0$. Moreover, it has positive entropy, as otherwise $X_0\in L^2(\Pi(\X))$, whence $X_0=0$.  By the Rokhlin-Sinai theorem, see e.g.\ \cite{Wa} Thm.\ 4.36, the spectral measure $\sigma_{X_0}$ of a Sarnak process $\X$, i.e.\ the (symmetric, Borel, positive) measure determined by $\widehat{\sigma}_{X_0}(n)=\EE X_n\ov{X}_0$, $n\in\Z$, has to be absolutely continuous.  Therefore, if the spectral measure of a stationary process is not absolutely continuous, the process is not Sarnak. To see positive entropy processes whose spectral measure is partly singular notice that (contrary to the multiplication, cf.~\Cref{theo:t:mnozenie}), Sarnak processes are not stable under addition of deterministic processes. Indeed, suppose that $\X$ is a Sarnak process, and let $\Y=(Y_n)$ be a (finite-valued) centered, stationary process with zero entropy such that the dynamical systems generated by these processes are disjoint.\footnote{For example, we can take for $\X$ a Kolmogorov process and for $\Y$ a finite valued process representing an irrational rotation on the circle.} Then, in view of \cite{Fu}, the process $\X+\Y=(X_n+Y_n)$ is generating for the Cartesian product of the dynamical systems corresponding to  $\X$ and $\Y$, so it has positive entropy. On the other hand, the process $\X+\Y$ is not Sarnak, as $X_0+Y_0$ is not orthogonal to $Y_0$, so it is not orthogonal to $L^2(\Pi(\X+\Y))$. Moreover, the spectral measure of the process $\X+\Y$ is $\sigma_{X_0}+\sigma_{Y_0}$. Varying $\Y$, we can now obtain a positive entropy non-Sarnak process whose spectral measure is partly singular, as well as a process whose spectral measure is Lebesgue.

    \begin{exam}{When factors of a Sarnak process are Sarnak?}{}
       Assume that $\X$ is a Sarnak process. One can ask which functions (depending on finitely\footnote{If we admit infinitely many coordinates then the Pinsker factor can be expressed as a function of
       a generating process.} many coordinates) of this process yield Sarnak processes? To get a partial answer to this question, assume that $f:\R^k\to\R$ is measurable and let $\Y=(Y_n)$ with $$Y_n:=f(X_n,X_{n+1},\ldots,X_{n+k-1}), \;n\in\Z.$$ Note that
       $$
          \sigma(Y_{-n},Y_{-n-1},\ldots)\subset \sigma(X_{-n+k-1},X_{-n+k-2},\ldots),$$ whence $\Pi(\Y)\subset \Pi(\X)$. If the function
       $f$ is linear then $Y_0\perp L^2(\Pi(\X))$. Hence, under linearity, $\Y$ is Sarnak. A similar argument shows that, for any $k \ge 1$, the jumping process $(X_{kn})_n$ is Sarnak.

    \end{exam}

    \begin{rema}{}{}
       A large source of examples of Sarnak  processes is given by the replacement of the orthogonality requirement in \Cref{defi:d:SP} with the independence condition: $X_0 \amalg \Pi(\X)$. Note that each $\pm1$-valued Sarnak process satisfies this property. Take, for example, any case in which  $\Pi(\X)$ is trivial, i.e.\ $\X$ is Kolmogorov. Then $\X$ is  Sarnak as soon as it is centered.
    \end{rema}

    \begin{exam}{Non-ergodic Sarnak process}{ex:non_erg}
       Take the Bernoulli process $\mc{B}\left(1/2, 1/2\right)$ on $\{-1,1\}^{\Z}$ and let $\bfu$ be any generic sequence for it. Let $\bfv\in\{0,1\}^{\N}$ be any generic point for a zero entropy measure. Then, in view of \Cref{theo:t:mnozenie},  $\bfu\bfv$ generates a Sarnak process and since $(\bfu \bfv)^2=\bfv$, clearly, the zero entropy system is a factor of the system given by $\bfu\bfv$, and the zero entropy system need not be ergodic.   In fact, as proved in  \cite{Ku-Le(jr)} (using \cite{Fu-Pe-We}), once $\X$ is Kolmogorov, and $\Y$ is deterministic with $\mathcal{Y}=\{0,1\}$, we can always filter out the zero entropy system from the dynamical system given by $\X\Y$.\footnote{
          Assuming the Chowla conjecture, the above applies to $\mob=\lio\cdot\mob^2$ as $\mob^2$ is generic for a zero entropy measure. As Sarnak's conjecture for $\mob$ and $\lio$ are equivalent,  e.g. \cite{Fe-Ku-Le}, the processes determined by $\lio$ are Sarnak iff the processes determined by $\mob$ so are.}\end{exam}

    \begin{exam}{Sarnak process with non-trivial Pinsker $\sigma$-algebra}{e:p2} We will show that there are $\pm1$-valued Sarnak processes with non-trivial remote past.

       Consider the space $\{0,1\}$ with $\mc{B}\left(1/2\right)$ measure and $Ti= 1 - i$, and the full shift $S$ on $\{-1,1\}^{\Z}$ considered with the Bernoulli measure $\mc{B}\left(1/2, 1/2\right)$. Let $\widetilde{T}$ be the corresponding skew product (considered with the product measure)
       $$
          \widetilde{T}(i,\x):=(1 -i ,S^{i}\x).$$
       Clearly, $\widetilde{T}$ has $T$ as factor, so its Pinsker algebra is non-trivial.
       Let $Y_0(i,\x):=x_0$. We claim that the process $(Y_n)$, with $Y_n:=Y_0\circ \widetilde{T}^n$, is Sarnak.
    \end{exam}
    \begin{rema}{}{}
       Note that $T$ produces sequences of alternating $0$'s and $1$'s. In the preceding example, we have $I \sim \mc{B}(1/2) \amalg \boldsymbol{B}\sim \mc{B}(1/2, 1/2)$. If $\boldsymbol{Z}$ is given by $Z_i = B_{\lfloor i /2 \rfloor}$, then $\Y$ is given by  $Y_i = Z_{i + I}$ for 
       every $i\in\Z$.
    \end{rema}
    \begin{proof}[Proof of the statement from \Cref{exam:e:p2}]
       We need three facts:
       \begin{align}
          \label{ps1}
           & Y_0\perp L^2(\sigma(I)) \\ 
          \label{ps2}
           & (Y_n)\text{ generates the full $\sigma$-algebra}                 \\
          \label{ps3}
           & \Pi(\widetilde{T})=\text{the first coordinate $\sigma$-algebra}.
       \end{align}
       While \eqref{ps1} is obvious, to establish~\eqref{ps2} note that the consecutive values of the process at $(0,x)$ are:
       $$
          Y_0(0,x)=x_0,Y_1(0,x)=x_0,Y_2(0,x)=x_1,Y_3(0,x)=x_1, Y_4(0,x)=x_2,
          Y_5(0,x)=x_2,\ldots,
       $$
       so it is clear that the process separates points and therefore~\eqref{ps2} holds. Finally, note that $\widetilde{T}^2$ is the Cartesian product of two factors: of the fist coordinate $\sigma$-algebra (on which that action is the identity) and of the second coordinate $\sigma$-algebra on which it acts as the shift $S$ (indeed, $\widetilde{T}(\{0,1\}\times\{x\})=
          \{0,1\}\times\{Sx\}$). The latter is just the Bernoulli $\mc{B}\left(1/2, 1/2\right)$, and since the Pinsker $\sigma$-algebra of the product is the product of Pinsker $\sigma$-algebras, we see that $\Pi(\widetilde{T}^2)$=first coordinate $\sigma$-algebra. To conclude, i.e.~to obtain~\eqref{ps3}, it is enough to notice that the Pinsker $\sigma$-algebras of non-zero powers of an automorphism are all equal.
    \end{proof}
    \begin{rema}{Combinatorial intuition for the above example}{}
       Suppose that $\bfu$ satisfies~\eqref{Sarnakg}. Set $\boldsymbol{v}$ via
       $$
          v_{2n}= v_{2n + 1} :=u_{n}\text{ for all }n\in\Z.$$
       Then $\bfv$ also satisfies~\eqref{Sarnakg}, as $T^2$ has also zero entropy,
       $$
          \frac1N\sum_{n\leq N}f(T^nx)v_n=\frac1N\sum_{m\leq N/2}u_m(f(T^{2m}x)+f(T^{2m+1}x))+o(1)\to0\text{ when }N\to\infty.$$

       It follows that the Furstenberg systems of $\bfv$ yield Sarnak processes. However, if we consider the subshift $X_{\bfv}:=\ov{\{S^n\bfv\colon n\in\Z\}}$ then for each  $y\in X_{\bfv}$ either \\(a) $y_{2n}=y_{2n+1}$ for all $n\in\Z$ holds or\\ (b) $y_{2n}=y_{2n-1}$ for all $n\in\Z$ is satisfied.\\
       Note that if a point $y$ satisfies (a) and (b) simultaneously, then $y$ is a fixed point. Notice that no Furstenberg system of $\bfv$ gives positive measure to the set of fixed points.\footnote{This is a property of all Sarnak processes not taking the value zero. Indeed, let $F$ be the set of fixed points of the process $\X$ (it is a finite  set naturally identified with a subset of $\mathcal{X})$. Of course this set belongs to $\Pi(\X)$, moreover, each subset of $F$ is $\Pi(\X)$-measurable. It follows that the function $X_0\cdot \raz_{F}$ is $\Pi(\X)$-measurable.
          Hence $\EE(X_0\cdot \ov{(X_0\cdot\raz_F)})=0$. On the other hand, this integral equals $\int_F|X_0|^2>0$, a contradiction.} Then if $A\subset X_{\bfv}$ denote the set of point satisfying (a) and $B$ stands for the set of points satisfying (b), then $A\cap B=\emptyset$, $B=SA$, whence -1 is an eigenvalue of the dynamical system given by any Furstenberg system of $\bfv$, whence the remote past cannot be trivial.
    \end{rema}

\section{Proofs} \label{s:BasicProp}
    The organization is as follows. Each section containing a proof of one of our results is preceded by
   a section which provides some necessary background.
    \subsection{Background for \Cref{theo:p:alaweakB}}
        In this part all random variables take only \textbf{finitely many values}. For any random variable $X$, we denote by $p_X$ its distribution. The main goal of this section is to prove the following fact.
        \begin{lemm}{Pinsker two-sided inequality in finite state case}{l:mich}
           Fix random variables $X, Y$. Define  $1/\beta := \sup_{x, y}\frac{p_{X, Y}(x, y)}{p_X(x)p_Y(y)}$. Then
           \begin{equation}\label{it7}
              \frac{2\sqrt{\beta}}{\log e}\left[\mathbb{H}(X) - \mathbb{H}(X|Y)\right] \le			\E\norm[p_{X | Y}(\cdot | Y) - p_X] \le \sqrt{2\left[\mathbb{H}(X) - \mathbb{H}(X|Y)\right]}.
           \end{equation}
        \end{lemm}
        Note that, as a direct result, we immediately obtain the following fact.
        \begin{coro}{Pinsker inequality for random sign}{pinsker inequality}
           Fix random variables $X, Y$ such that $X$ is a symmetric random sign. Then
           \begin{equation}\label{pinsker inequality}
              \frac{\sqrt{2}}{\log e}\left[1 - \mathbb{H}(X|Y)\right] \le	\E\norm[p_{X | Y}(\cdot | Y) - p_X] \le \sqrt{2\left[1 - \mathbb{H}(X|Y)\right]}.
           \end{equation}
        \end{coro}
        \begin{proof}
           If $X$ is a random symmetric sign (i.e.\ $X$ takes values $\pm1$ with equal probability) then $p_X \equiv 1/2$ and thus $\mathbb{H}(X) = \log 2 = 1$ and
           \begin{equation}
              1 / \beta = 2\sup_{x, y} p_{X|Y}(x|y) \le 2.
           \end{equation}
        \end{proof}

        The idea of the proof  \Cref{lemm:l:mich} is to use the Pinsker and reversed Pinsker inequalities.
        In order to do so, we need reformulate quantities appearing in \eqref{pinsker inequality}  in terms of KL-divergnece and get rid of the integral $\E$. To this end, recall that given two probability distributions  $p$ and $q$ on a finite space $\mc{X}$,
        \begin{equation}
           \di[p][q] = \sum_{x \in \mc{X}} p(x)\log\frac{p(x)}{q(x)}
        \end{equation}
        stands the KL-divergence between distributions $p_X$ and $p_Y$ with the convention that $0/0 = 0$. Now, one easily checks that
        \begin{equation}
           \di[p_{X, Y}][p_X\otimes p_Y] = \h[X] - \coh[X][Y]
        \end{equation}
        The last ingredient needed for the proof is the following one.
        \begin{prop}{Mean conditional variation norm equals to joint total variation}{mean v}
           For any random variables $X, Y$,
           \begin{equation}\label{it4}
              \E\norm[p_{X | Y}(\cdot | Y) - p_X] = \norm[p_{X,Y} - p_X\otimes p_Y]
           \end{equation}
        \end{prop}
        \begin{proof}
           By the very definitions,
           \begin{align*}
              \E\norm[p_{X | Y}(\cdot | Y) - p_X] & = \E_{p_Y} \norm[p_{X | Y}(\cdot | Y) - p_X]= \sum_{y \in \mathcal{Y}} p_Y(y) \sum_{x \in \mathcal{X}} |p_{X|Y}(x|y) - p_{X}(x)| \\
                                             & =\sum_{x, y}|p_{X,Y}(x,y) -p_X(x)p_Y(y)| = \norm[p_{X,Y} - p_X\otimes p_Y].
           \end{align*}
        \end{proof}
        \begin{proof}[Proof of \Cref{lemm:l:mich}]
           The upper bound is obtained using the Pinsker inequality. Indeed, by \Cref{prop:mean v},
           \beq\label{it5}
           \begin{array}{c}\E\norm[p_{X | Y}(\cdot | Y) - p_X]  = \|p_{X,Y} - p_X\otimes p_Y\|_{TV} \le  \sqrt{2\mathbbm{D}(p_{X,Y}||p_X\otimes p_Y)}=
              \sqrt{2\left[\mathbb{H}(X) - \mathbb{H}(X|Y)\right]}.\end{array}
           \eeq
           On the other hand, using the reverse Pinsker inequality, see e.g.\ \cite{verdu} Theorem 7, we obtain
           \begin{equation*}
              \E\norm[p_{X | Y}(\cdot | Y) - p_X] \ge  \frac{2\sqrt{\beta}}{\log e} \mathbbm{D}(p_{X,Y}||p_X\otimes p_Y) = \frac{2\sqrt{\beta}}{\log e} \left[\mathbb{H}(X) - \mathbb{H}(X|Y)\right],
           \end{equation*}
           where $1/\beta := \sup_{x, y}\frac{p_{X, Y}(x, y)}{p_X(x)p_Y(y)}$.
        \end{proof}

        Let us now make a remark on the monotonicity of the three terms which appear in \eqref{pinsker inequality}. It is well known that given two probability distributions on the same state space $\mc{X}$,
        \begin{equation}\label{diagonal realization}
           \norm[p - q] = \inf_{X \sim p,\; Y \sim q}\pof[X \neq Y]
        \end{equation}
        Thus, intuitively, $\norm[p - q]$ tells us what is the best (that is, closest to diagonal) coupling of $p$ and $q$.
        This fact immediately yields monotonicity property of the total variation norm.
        \begin{prop}{ Monotonicity property of the total variation norm}{}
           Let $X, Y, Z \in \mc{X}$ be some random variables. Then
           \begin{equation}
              \norm[p_{X,Y,Z}-p_X\ot p_{Y,Z}] \ge \norm[p_{X,Y}-p_X\ot p_{Y}].
           \end{equation}
        \end{prop}
        \begin{proof}
           Take the realization of \eqref{diagonal realization} with $p = p_{X,Y,Z}$ and $q = p_X\ot p_{Y,Z}$ to obtain random variables $(X, Y, Z) \sim p_{X, Y, Z}$ and $(X', Y', Z') \sim p_X \ot p_{Y, Z}$ such that $\pof[X \neq X', Y \neq Y', Z \neq Z'] = \norm[p_{X,Y,Z}-p_X\ot p_{Y,Z}]$. Clearly,
           \begin{equation}
              \pof[X \neq X', Y \neq Y', Z \neq Z'] \le \pof[X \neq X', Y \neq Y']
           \end{equation}
           and thus, once more using \eqref{diagonal realization},
           \begin{equation*}
              \norm[p_{X,Y}-p_X\ot p_{Y}] \le \pof[X \neq X', Y \neq Y'] \le \pof[X \neq X', Y \neq Y', Z \neq Z'] = \norm[p_{X,Y,Z}-p_X\ot p_{Y,Z}].
           \end{equation*}
        \end{proof}
        \begin{rema}{Monotonicity of terms in \eqref{pinsker inequality}}{}
           Since the Shannon entropy is monotonic, we have
           \begin{equation}
              \h[X] - \coh[X][Y] \le \h[X] - \coh[X][Y, Z]
           \end{equation}
           Combining this observation with the previous remark, we see that both terms in \eqref{pinsker inequality} $\E\norm[p_{X | Y}(\cdot | Y) - p_X]$ and  $1 - \mathbb{H}(X|Y)$ posses monotonicity property: if $Y$ is a vector random variable in \eqref{pinsker inequality} then dropping of any of its coordinates forces all terms in \eqref{pinsker inequality} to decrease.
        \end{rema}

        Last but not least let us present the well known fact which says that in case of a random variable attaining only two values and being $L^2$ perpendicular to some $\sigma$-algebra is equivalent to being independent of that algebra. Note that this observation can be used for Sarnak processes $\X$ taking at most two values, namely, in such a case $X_0 \amalg \Pi(\X)$.
        \begin{prop}{}{uncor}
           Let $X$ be a two valued random variable and $\mathcal{G}$ be some sub-$\sigma$-algebra such that $\cmv[X][\mathcal{G}] = \E X$. Then $X \amalg \mathcal{G}$.
        \end{prop}
        \begin{proof}
           Clearly, our assumption is equivalent to: $\E X \iof[G] = \E X \E \iof[G]$ for any $G \in \mc{G}$. In other words, $Cov(X, \iof[G]) = 0$. Now, it is enough to recall standard fact that if two random variables taking only two values are uncorrelated then they are independent. Applying this fact, we get that for any $G \in \mc{G}$, $X \amalg \iof[G]$. Hence $X \amalg \mc{G}$.
        \end{proof}
        \begin{rema}{When uncorrelated random variables are independent}{}
           In the above proof we used the fact that if $X\in \mc{X}$ and $Y \in \mc{Y}$ are uncorrelated and $|\mc{X}| = |\mc{Y}| = 2$ then they are independent. Here we provide a sketch of a proof. Firstly, without loss of generality we can assume that $\mc{X} = \mc{Y} = \{0, 1\}$. To see this consider
           \begin{equation}
              X' = \frac{X - x_0}{x_1 - x_0}, \quad Y' = \frac{Y - y_0}{y_1 - y_0},
           \end{equation}
           where $\mc{X} = \{x_0, x_1\}$ and $\mc{Y} = \{y_0, y_1\}$. For binary random variables, the assumption is equivalent to
           \begin{equation}
              \pof[X = 1, Y = 1] = \pof[X = 1]\pof[Y = 1].
           \end{equation}
           It remains to use the fact that if events $A$ and $B$ are independent then so are $A^c$ and $B$.
        \end{rema}
    \subsection{Proof of \Cref{theo:p:alaweakB}}
        Firstly, since $X_0$ is perpendicular to $\Pi(\X)$ and $X_0$ takes only two values, by \Cref{prop:uncor}, we in fact have $X_0 \amalg \Pi(\X)$. In particular, Sarnak property is equivalent to the $\coh[X_0][X_{\le -g}] \xrightarrow{g\to\infty} \h[X_0] = 1$. By the monotonicity property of conditional entropy this is equivalent to the statement that for every $\vep >0$ there is a gap $g \ge 1$ such that for all $m \in \N$
        \begin{equation}\label{equiv}
           1 - \coh[X_0][X_{-g - m}^{-g}] \le \vep.
        \end{equation}
        Indeed, if \eqref{equiv} holds then taking $m\to\infty$ yields the result. Conversely, for every $\vep$ there is $g \in \N$ such that $1 - \coh[X_0][X_{\le -g}] \le \vep$. However, for any $m \ge 0$, $1 - \coh[X_0][X_{\le -g}] \ge 1 - \coh[X_0][X_{-g - m}^{-g}]$.

        Now, an application of \Cref{coro:pinsker inequality} with $X = X_0$ and $Y = X_{-m -g}^{-g}$ yields,
        \begin{equation}\label{it8}
           \frac{\sqrt{2}}{\log e}\left[1-\mathbb{H}(X_0|X_{-m - g}^{-g}) \right] \le\\			\E|p_{X_0 |X_{-m - g}^{-g} }(\cdot | X_{-m - g}^{-g}) - p_{X_0}| \le \sqrt{2\left[1-\mathbb{H}(X_0|X_{-m - g}^{-g}) \right]}
        \end{equation}
        and the result follows.
        \begin{rema}{Rate of convergence}{}
           Note that, thanks to the \eqref{it8}, if we know the rate of convergence for  $1-\mathbb{H}(X_0|X_{-m - g}^{-g}) $ then \eqref{it8} controls the rate for the term $	\E|p_{X_0 |X_{-m - g}^{-g} }(\cdot | X_{-m - g}^{-g}) - p_{X_0}|$.
        \end{rema}
        \begin{proof}[Proof of \Cref{coro:t:main}]
           Let $\mbox{X} \sim \kappa$,
           Since  $1/ n_k \sum_{0 \le n\leq n_k - 1}\delta_{S^n\bfu}\overset{k \to \infty}{\Rightarrow}\kappa$,
           \begin{equation}\label{g1}
              \lim_{k\to\infty}\frac1{n_k}|A(q,n_k)| = \lim_{k\to\infty}\frac1{n_k}\sum_{i = 0}^{n_k - 1} \iof[S^i{\boldsymbol{u}}\text{ starts with } q] = \kappa(q).
           \end{equation}
           Similarly,
           \begin{equation}\label{g2}
              \frac{1}{n_k}\left|\{n + g\in A(q,n_k) \mid u_n =1\}\right| = \frac1{n_k}\sum_{i = 0}^{n_k - 1} \iof[S^i\boldsymbol{u}(0) = 1]\iof[S^{i + g}\boldsymbol{u}\text{ starts with } q] \to \pof[X_0 = 1, X_g^{g + m} = q].
           \end{equation}
           Combining \eqref{g1} and \eqref{g2},
           $$\vep_{g}(q)= \left|\frac{1}{2}\pof[X_g^{g + m} = q] - \pof[X_0 = 1, X_g^{g + m} = q]\right|$$
           and the result follows from \Cref{theo:p:alaweakB}.
        \end{proof}
        \begin{rema}{}{}
           In the above proof we have used the symmetric version of \Cref{theo:p:alaweakB} with the past replaced by the future, which is valid as well because in the finite case the past tail $\sigma$-field of process equals to the future tail $\sigma$-field one. In particular, the  Sarnak property can be stated equivalently with reversed roles of the past and the future.
        \end{rema}

    \subsection{Background for \Cref{theo:t:mnozenie}}
        \begin{rema}{Relatively independent coupling above a $\sigma$-field}{}
           Later on we use the following construction. Given a probability space $\zdk$ and a sub-$\sigma$-algebra $\mathcal{E}\subset\mathcal{D}$, the formula
           $$
              \lambda(D_1\times D_2):=\int \EE(D_1|\mathcal{E})\EE(D_2|\mathcal{E})\,d\kappa$$
           determines a coupling on the space $(Z_1\times Z_2,\mathcal{D}_1\otimes\mathcal{D}_2)$  ($(Z_j,\mathcal{D}_j)=(Z,\mathcal{D})$ for $j=1,2$). A characteristic property of this coupling is that if we have two measurable functions $f,g:Z\to\R$ 
           then the following conditions are equivalent:
           \begin{description}
            \item[(i)] $f(z_1)=g(z_2)$ for $\la$-a.e.\ $(z_1,z_2)\in Z_1\times Z_2$,
            \item[(ii)]$f\overset{\kappa}{=}g$, and $f$ is $\mathcal{E}$-measurable.
           \end{description}
           In the following theorem  a similar type of joining is used,  for more information we refer the reader to \cite{Gl}, Examples~6.3 and Theorem~6.8.
        \end{rema}
        To prove \Cref{theo:t:mnozenie}, let us recall the classical theorem (see e.g.\ \cite{Ru}, the result is a consequence of the basic lemma on non-disjointness proved in \cite{Gl-Th-We} and \cite{Le-Pa-Th})  about joinings with deterministic systems.

        \begin{Th}\label{t:character}Let  $(Y,\mathcal{C},\nu,S)$ be a dynamical system and $(Z,\mathcal{D},\kappa,R)$ has entropy zero. Assume that $\rho\in J(S,R)$ is a joining of $\nu$ and $\kappa$. Then
           \beq\label{char1}
           \int_{Y\times Z}f(y)g(z)\,d\rho(y,z)=\int_{(Y/\Pi(S))\times Z}\EE(f|\Pi(S))(x)g(z)\,d\rho|_{(Y/\Pi(S))\times Z}(x, z),\eeq
           i.e.\ each such joining has to be the relatively independent extension of its restriction to $(Y/\Pi(S))\times Z$.\end{Th}

        \noindent

    \subsection{Proof of \Cref{theo:t:mnozenie}}

 By assumption, the dynamical system $\bigl(\mathcal{X}^{\Z}\times\mathcal{Y}^{\Z},\rho,S\bigr)$ associated to the stationary process $(\X,\Y)$ is a joining of the dynamical systems $(\mathcal{X}^{\Z},\nu,S)$ and $(\mathcal{Y}^{\Z},\kappa,S)$ given by $\X$ and $\Y$, respectively. Let us fix $F\in L^2(\Pi(\rho))$.
        Note that $\rho$ induces a joining of $(\mathcal{X}^{\Z},\nu,S)$ with the zero-entropy system $((\mathcal{X}^{\Z}\times \mathcal{Y}^{\Z})/\Pi(\rho), \rho|_{\Pi(\rho)},S\times S)$. 
        Therefore, by~\eqref{char1}, and using the assumption that $\X$ is Sarnak, we get
        $$
           \int X_0(\x)F(\x,\y)\,
           d\rho(\x,\y)=
           \int \EE\bigl(X_0|\Pi(\nu)\bigr)(\x)F(\x,\y)
           \,d\rho(\x,\y)=0.$$
        Since $Y_0\in L^2(\Pi(\rho))$, we also have
        $$\int X_0(\x)\Big(Y_0(\y)F(\x,\y)\Big)\,
           d\rho(\x,\y)=0, $$
        which shows that $X_0\ot Y_0\perp L^2(\Pi(\rho))$.

    \subsection{Background for \Cref{theo:p:thierry}}

        In  order to prove ~\Cref{theo:p:thierry}, we will need the result from \cite{Fu-Pe-We} and \cite{Ku-Le(jr)} which was already employed in \Cref{exam:ex:non_erg}, and whose proof  (we provide it for the sake of completeness) in case of $\{0,1\}$-valued processes  is a short compilation of the arguments from the aforementioned papers.

        \begin{lemm}{Retrieval of a deterministic process from its product with a K-process}{filtering}
           Let $(\X, \Y)$ be a stationary process such that both $\X$ and $\Y$ are binary, $\X\neq\boldsymbol{0}$ is Kolmogorov and  $\Y$ is deterministic (note that this implies $\X \amalg \Y$). Then $\Y$ is measurable with respect to $\sigma(\X\Y)$.
        \end{lemm}
        \begin{proof}
           Consider two copies $(\X', \Y')$ and $(\X'', \Y'')$ of $(\X, \Y)$, which are relatively independent over $\X\Y$. Then, $X'_nY'_n = X''_nY''_n$ for all $n\in\Z$. Multiply both sides of this equality by $\iof[Y'_n = 0]\iof[Y''_n = 1]$ to obtain $ 0 = X''_n\iof[Y'_n = 0]\iof[Y''_n = 1]$. Since $(\Y, \Y')$ is deterministic and $\X''$ is Kolmogorov, it follows that these two processes are disjoint. In particular, $X''_n \amalg \iof[Y'_n = 0]\iof[Y''_n = 1]$. Hence $0 = \pof[Y'_n = 0, Y''_n = 1]\E X''_n$ which implies that (recall $\X \neq \boldsymbol{0}$) $0 = \pof[Y'_n = 0, Y''_n = 1]$. By symmetry it follows that $Y'_n = Y''_n$. Thus, $\Y' = \Y''$  which concludes the proof.
        \end{proof}
        \begin{rema}{Random variables concentrated on the graph of a function}{graph}
           The following simple observation allows one to generalize \Cref{lemm:filtering} to the case in which $\X$ is not necessarily Kolmogorov. Suppose that $(Y, Z) \sim (X, f(X))$ for some random variables $X, Y, Z$ and some measurable function $f$. Then $Z = f(Y)$. Indeed, we have $p_{Z,Y}(Z\neq f(Y))=p_{X,f(X)}(f(X)\neq f(X))=0$
           and the result follows.
        \end{rema}

        \begin{lemm}{Modify and retrieve}{modify and retrieve}
           Let $(\X, \Y)$ be a stationary process such that both $\X \neq \boldsymbol{0}$, $\Y$ are binary and $\Y$ is deterministic. Suppose additionally that we can find a Kolmogorov binary process $\X' \amalg \Y$ such that $ \X'\Y \sim \X \Y$ conditionally on $\Y = \y$. Then $\Y$ is measurable with respect to $\sigma(\X\Y)$.
        \end{lemm}
        \begin{proof}
           By \Cref{lemm:filtering} we know that there is a function such that $\Y = f(\X'\Y)$. Moreover, $(\X\Y, \Y) \sim ( \X'\Y, \Y) = (\X'\Y, f(\X' \Y))$. It remains to use \Cref{rema:graph} with $X = \X'\Y$, $Y = \X\Y$ and $Z = \Y$.
        \end{proof}
        Let us present now a simple observation needed in the proof of the  lemma below.
        \begin{prop}{Example of a factor of a Bernoulli which is Bernoulli}{bernoulli factor}
           Suppose that $\X$ is a Bernoulli $\mc{B}(1/2, 1/2)$ process with two point state space $\mc{X}$. Then $\Y$ given by
           \begin{equation}
              Y_i = \iof[X_i \neq X_{i + 1}]
           \end{equation}
           is Bernoulli $\mc{B}(1/2, 1/2)$  as well.
        \end{prop}
        \begin{proof}
           Since $\Y$ is binary it is enough to check if
           \begin{equation}
              \pof[Y_0^{n - 1} = 1] = \prod_{i = 0}^{n - 1}\pof[Y_i = 1] = 2^{-n}.
           \end{equation}
           By the very definition,
           \begin{equation}
              \pof[Y_0^{n-1} = 1] = \pof[\forall_{0 \le i \le n - 1}\; X_i \neq X_{i + 1} ].
           \end{equation}
           But there are only two alternating sequences (on $\mc{Y}^{n + 1}$) of length $n + 1$. The result follows.
        \end{proof}
        \begin{lemm}{Skew process $\X'$}{plateau}
           Let $\X \amalg \Y$ be stationary processes such that $\X$ is i.i.d., $X_i \in \mc{X} = \{a, b\}$ ($a, b\in \C$, $a \neq b$) and the binary process $\Y$ is deterministic.
           Let $\boldsymbol{K}=(K_n)_{n\in\Z}$ be the non-decreasing sequence of random times defined by
           $K_0 := 0$, and for $n\geq1$, $K_n := \sum_{i = 0}^{n - 1} Y_i$ and $K_{-n} := -\sum_{i = 1}^{n} Y_{-i}$ (so that $K_{n+1}=K_n+Y_n$ for all $n$). We assume that we have almost surely
           $$ \lim_{n\to \infty} K_{-n}=-\infty \text{ and }\lim_{n\to \infty} K_{n}=\infty.$$
           We define the process $\X'$ via $X'_i = X_{K_i}$. Then $(\X', \Y)$ is stationary, and
           \begin{equation}\label{pinskers}
              \Pi(\X', \Y) = \Pi(\X') = \sigma(\Y).
           \end{equation}
        \end{lemm}
        \begin{rema}{Why the name skew process?}{}
           Clearly, by the very definition, we can look at $\X'$ as a version of $\X$ skewed by $\Y$. Moreover, $\X'$ arises as a coordinate in a skew product of dynamical systems (see \eqref{skew} below, which also justifies that $(\X', \Y)$ is stationary).
        \end{rema}
        \begin{rema}{$\X'$ is Sarnak}{recipe}
           The above lemma is a recipe for producing Sarnak processes: indeed the centered skew process $\X' - \E X_0$ is Sarnak. To see this, note that $X'_0 = X_0$, and
           \begin{equation}
              \cmv[X'_0][\Pi(\X')] \overset{\eqref{pinskers}}{=} \cmv[X_0][\Pi(\Y)] \overset{\X \amalg \Y}{=} \E X_0.
           \end{equation}
        \end{rema}
        \begin{proof}[Proof of \Cref{lemm:plateau}]
           Firstly, we show $\Pi(\X', \Y) = \sigma(\Y)$. Clearly, $\Pi(\X', \Y) \supset \Pi(\Y)=\sigma(\Y)$. Hence it remains to show that $\Pi(\X', \Y) \subset \sigma(\Y)$. Intuitively this is clear because $\Pi(\X)$ is trivial so as it comes to some remote past of $(\X', \Y)$, the past of $\X'$ brings no additional information to the past of $\Y$. Formally, take some random variable $F$ measurable with respect to $\Pi(\X', \Y)$. By the definition of the tail $\sigma$-algebra, for every $n\in\N$, there exists a measurable function $f_n$ such that
           \begin{equation}
              F = f_n(X'_{\le -n}, Y_{\le -n}).
           \end{equation}
           Consider this equation conditionally on $\Y = \y$. Then, by the very definition of $\X'$, there exists $f_{n, \y}$ (naturally defined via bijection correspondence $f_{n, \Y}(X_{\le K_{-n}}) =f_n(X'_{\le -n}, Y_{\le -n})$) such that
           \begin{equation}\label{tail}
              F  = f_{n, \y}(X_{\le k_{-n}}),
           \end{equation}
           where $\boldsymbol{k}$ is the value of $\boldsymbol{K}$ under $\Y = \y$. However, $k_{-n}\xrightarrow{n \to\infty} -\infty$. Since $\X \amalg \Y$ and
           \eqref{tail} holds for any $n$, we conclude that conditionally on $\Y = \y$, $F \in \Pi(\X)$. However, $\Pi(\X)$ is trivial and hence conditionally on $\Y = \y$, $F$ is a constant. Therefore, (unconditionally) $F$ is $\sigma(\Y)$-measurable.

           Now, we take care of $ \Pi(\X') = \sigma(\Y)$. By the previous step, $ \Pi(\X') \subset \Pi(\X', \Y) = \sigma(\Y)$. It remains to show that $\Pi(\X') \supset  \sigma(\Y)$. To this end, define a process $\boldsymbol{C}$ which checks when $\X'$ changes, that is,

           \begin{equation}
              C_n := \iof[X'_{n- 1} \neq X'_n].
           \end{equation}
           Note that, if we could recover the process $\Y$ from $\boldsymbol{C}$, then it would be the end of the proof. Indeed, in such a case,  $\boldsymbol{C}$ is a function of $\X'$ and thus  $\Y = f(\X')$. 
           To achieve this, using the fact that $C_n = 1$ iff $Y_n = 1$ and $X_{K_n - 1} \neq X_{K_n}$, we can express $C_n$ as
           \begin{equation}
              C_n = Y_n\iof[X_{K_n - 1} \neq X_{K_n}].
           \end{equation}
           We claim that, conditionally on $\Y = \y$, the process  $\boldsymbol{C}$ is distributed as $\boldsymbol{B}\Y$ where $\mc{B}(1/2, 1/2) \sim \boldsymbol{B} \amalg \Y$. For this, for a fixed $n\in\N$, and fixed $y,w\in\{0,1\}^\Z$, we compare $\pof[ C_{-n}^n=w_{-n}^n\,|\, Y=y]$ and $\pof[ (BY)_{-n}^n=w_{-n}^n\,|\, Y=y]$. Clearly, both vanish if there exists $j\in\{-n,\ldots,n\}$ such that $y_j=0$ but $w_j=1$. Otherwise, let
           $$ \{j_1<j_2<\ldots<j_\ell\} :=  \bigl\{j\in\{-n,\ldots,n\}:\ y_j=1\bigr\}. $$
           Since we assume now that $w_{-n}^n\leq y_{-n}^n$, we have
           $$ \pof[(BY)_{-n}^n=w_{-n}^n\,|\, Y=y] = \pof[B_{j_i}=w_{j_i}\text{ for all }i\in\{1\ldots,\ell\}] = \frac{1}{2^\ell}. $$
           On the other hand, denoting again by $\boldsymbol{k}$ the value of $\boldsymbol{K}$ under $\Y = \y$, since $k_{j_1}<k_{j_2}<\ldots<k_{j_\ell}$, by \Cref{prop:bernoulli factor} we also have
           $$ \pof[C_{-n}^n=w_{-n}^n\,|\, Y=y] = \pof[\iof[X_{k_{j_i} - 1} \neq X_{k_{j_i}}]=w_{j_i}\text{ for all }i\in\{1\ldots,\ell\}] = \frac{1}{2^\ell}. $$
           This completes the proof of the claim, and then it is enough to apply \Cref{lemm:modify and retrieve} to conclude.
        \end{proof}
    \subsection{Proof of \Cref{theo:p:thierry}}
        By a non-ergodic version of Jewett-Krieger theorem (see \cite{Al-Se}, Theorem 1.2), each aperiodic, zero entropy system can be realized as a binary deterministic process.
        Thus, in view of \Cref{rema:recipe}, to finish the proof, it is enough to construct an appropriate processes from \Cref{lemm:plateau}. To do so consider the full shift $S$ on $\{-1,1\}^{\Z}$ with the Bernoulli measure $\nu:=\mc{B}\left(1/2, 1/2\right)$ and some 
    aperiodic zero-entropy system $(\{0,1\}^\Z,\mu,S)$. 
    On the product space $\{-1,1\}^{\Z}\times \{0,1\}^\Z$, let $\overline{T}$ be the corresponding skew product (considered with the product measure $\nu\otimes \mu$):
        \begin{equation}\label{skew}
           \overline{T}(\x,\y):=(S^{y_0}\x, S\y).
        \end{equation}

        Let $Y_n$, $n\in\Z$, (respectively $X_k$, $k\in\Z$) denote the projection on the $n$-coordinate on $\{0,1\}^\Z$ (respectively, on the $k$-coordinate on $\{-1,1\}^{\Z}$). Then, we construct the stationary process $\X'$ by setting $X'_n:=X_0\circ\overline{T}^n$.
        Note that for all $n\in\Z$ we can write $X'_n=X_{K_n}$, where $\boldsymbol{K}$ is $\Y$-measurable, and recursively defined by
        \begin{itemize}
           \item $K_0:=0$,
           \item For each $n\ge1$, $K_n:=K_{n-1}+Y_n$,
           \item For each $n\le -1$, $K_n:=K_{n+1}-Y_{n+1}$.
        \end{itemize}
        It remains to note that, by aperiodicity of the system we started from, we have $\mu(y_n=0\text{ for all }n)=0$. Therefore, for almost every ergodic component $\lambda$ of the system $(\{0,1\}^\Z,\mu,S)$, we have $\lambda(y_0=1)>0$, and then by the pointwise ergodic theorem we get that $K_n \xrightarrow{n\to \pm\infty} \pm\infty$ almost surely.

\subsection{Background for the proof of \Cref{theo:t:EDM}}
The following result has been proved in \cite{Go-Le-Ru}:

\begin{prop}{}{p:ed1} Let $R$ be an automorphism of $\zdk$. Let $\kappa=\int_{\Gamma}\kappa_\gamma\,dQ(\gamma)$ be the ergodic decomposition of $\kappa$. Then there exists an $R$-invariant $\sigma$-algebra $\mathcal{C}\subset\mathcal{D}$ such that $\Pi(\kappa)=\mathcal{C}$ $\kappa$-a.e.\ and for $Q$-a.a.\ $\gamma\in \Gamma$, we have $\Pi(\kappa_\gamma)=\mathcal{C}$ $\kappa_\gamma$-a.e. Moreover, for each $f\in L^1(\kappa)$, there exists a $\mathcal{C}$-measurable $g$ such that $\E_\kappa(f|\mathcal{C})=g$ $\kappa$-a.e.\ and, for $Q$-a.a.\ $\gamma\in \Gamma$, we have $\E_{\kappa_\gamma}(f|\mathcal{C})=g$ $\kappa_\gamma$-a.e.\end{prop}

Assume now that $\mathcal{X}\subset\C$ is finite and let $\mu$ be an $S$-invariant measure on $\mathcal{X}^{\Z}$. Then the coordinate projection process $(\pi_n)=((\pi_n),\mu)$ is stationary (and has distribution $\mu$). Assume that $\mu=\int_{\Gamma}\mu_\gamma\,dQ(\gamma)$ is the ergodic decomposition of $\mu$.

\begin{lemm}{}{l:ed2} The stationary process $((\pi_n),\mu)$ is Sarnak if and only if the stationary processes $((\pi_n),\mu_\gamma)$  are Sarnak for $Q$-a.e.\ $\gamma\in\Gamma$.\end{lemm}
\begin{proof}
$\Rightarrow$ We use \Cref{prop:p:ed1} for $(\mathcal{X}^\Z,\mu,S)$ and $f=\pi_0$.  We have
$$
0=\|\EE_\mu(\pi_0|\Pi(\mu))\|_{L^2}^2=\|\EE_\mu(\pi_0|\mathcal{C})\|_{L^2}^2=\int \EE_\mu(\pi_0|\mathcal{C})\ov{\pi}_0\,d\mu=\int g\ov{\pi}_0\,d\mu=
$$$$
\int_\Gamma\Big(\int g \ov{\pi}_0\,d\mu_\gamma\Big)\,dQ(\gamma)=
\int_\Gamma\Big(\int \EE_{\mu_\gamma}(g \ov{\pi}_0|\mathcal{C})\,d\mu_\gamma\Big)\,dQ(\gamma)=
$$$$
\int_\Gamma g\EE_{\mu_\gamma}(\ov{\pi}_0|\mathcal{C})\,d\mu_\gamma\Big)\,dQ(\gamma)=
\int_\Gamma\Big(\int |\EE_{\mu_\gamma}(\pi_0|\mathcal{C})|^2\,d\mu_\gamma\Big)\,dQ(\gamma),$$
whence, for $Q$-a.e.\ $\gamma$, we have
$$
0=\EE_{\mu_\gamma}(\pi_0|\mathcal{C})=\EE_{\mu_\gamma}(\pi_0|\Pi(\mu_\gamma)),$$
so our claim follows.

$\Leftarrow$  The same by reading in the reversed order.\footnote{Or using only the first property of $\mathcal{C}$: for each $h\in L^\infty(\mathcal{C})$,
$$
 \int \E_\mu(\pi_0|\mathcal{C})h\, d\mu=\int \pi_0 h\, d\mu=
\int_{\Gamma}(\int \pi_0 h\, d\mu_\gamma)\, dQ(\gamma)
 =$$$$ \int_{\Gamma}(\int \E_{\mu_\gamma}(\pi_0|\mathcal{C})h d\mu_\gamma)
 dQ(\gamma)=0$$
since $\EE_{\mu_\gamma}(\pi_0|\mathcal{C})=0$. Since $h$ was arbitrary, $\E_\mu(\pi_0|\mathcal{C})=\E_\mu(\pi_0|\Pi(\mu))=0$.}
\end{proof}

\subsection{Proofs of \Cref{theo:t:EDM} and of \Cref{coro:c:EDT}}
\Cref{theo:t:EDM} follows directly from \Cref{lemm:l:ed2}.

\noindent
{\em Proof of \Cref{coro:c:EDT}} Assume first that $\X$ (with distribution $\mu$) is ergodic. Then almost every realization  $(X_n(\omega))$ of the process is generic for the measure $\mu$. But since $\X$ is Sarnak, $X_0\perp \Pi(\mu)$, so the Veech condition is satisfied for the (unique) Furstenberg system of $\bfu=(X_n(\omega))$, and therefore by \cite{Ka-Ku-Le-Ru}, $\bfu$ is orthogonal to all deterministic sequences. If $\X$ is not ergodic then, by \Cref{theo:t:EDM}, we pass to ergodic components  to which we apply the above argument.


        \vspace{2ex}

        \noindent
        Faculty of Mathematics and Computer Science\\
        Nicolaus Copernicus University, \\
        Chopin street 12/18, 87-100 Toru\'n, Poland\\
        mlem@mat.umk.pl

        \vspace{2ex}

        \noindent
        Faculty of Physics, Astronomy and Informatics\\
        Nicolaus Copernicus University, \\
        Grudzi\c{a}dzka street 5, 87-100 Toru\'n, Poland\\
        m.lemanczyk@umk.pl

        \vspace{2ex}

        \noindent
        CNRS, Univ Rouen Normandie\\
        LMRS, UMR6085
        F76000 Rouen, France\\
        Thierry.de-la-Rue@univ-rouen.fr

\end{document}